\documentclass[12pt]{article}
\usepackage[T1]{fontenc}
\usepackage{amssymb}
\usepackage{amsmath}
\usepackage{amsfonts}
\usepackage{amsthm}
\usepackage{mathrsfs}
\usepackage{hyperref}
\usepackage[showonlyrefs]{mathtools}
\mathtoolsset{showonlyrefs}
\numberwithin{equation}{section}
\newtheorem{theorem}{Theorem}[section]
\newtheorem{thmx}{Theorem}
 
\newtheorem{lemma}{Lemma}[section]

\theoremstyle{remark}

\newtheorem{example}{Example}[section]

\def\R{\mathbb{R}}

\def\C{\mathbb{C}}
\def\N{\mathbb{N}}
\def\O{\mathcal{O}}
\def\re{\operatorname{Re}}

\def\length{\operatorname{length}}
\begin{document}
\title{Entire functions with separated zeros and $1$-points}
\author{Walter Bergweiler and Alexandre Eremenko\thanks{Supported by NSF grant DMS-1665115.}}
\date{}
\maketitle
\centerline{\emph{Dedicated to the memory of Stephan Ruscheweyh}}
\begin{abstract}
We consider transcendental entire functions of finite order for which the zeros and $1$-points are in 
disjoint sectors.  Under suitable hypotheses on the sizes of these sectors we show 
that such functions must have a specific form, or that such functions do not exist at all.

\medskip


\medskip

\end{abstract}
\section{Introduction and results}
Our starting point is the following result of Biernacki~\cite[p.~533]{Biernacki1929}.
\begin{thmx} \label{thm-bier}
There is no transcendental entire function of finite order for which
the zeros accumulate in one direction and the $1$-points accumulate in
a different direction.
\end{thmx}
Here we say that a set $\{ a_n\}$ of complex numbers
accumulates in one direction if there exists a ray such that
for every open sector bisected by this ray all but finitely many $a_n$
lie in this sector.

We will prove the following generalizations of Theorem~\ref{thm-bier}.
\begin{theorem} \label{thm1}
Let $S_0$ and $S_1$ be closed sectors in $\C$ satisfying $S_0\cap S_1=\{0\}$.
Let $\theta_j$ denote the opening angle of $S_j$ and suppose
that 
\begin{equation}\label{1a0}
\min\{\theta_0,\theta_1\}<\frac{\pi}{2} \quad\text{and}\quad \max\{\theta_0,\theta_1\}<\pi.
\end{equation}
Then there is no transcendental entire function of finite order for which
all but finitely many zeros are in $S_0$ while all but finitely many $1$-points are in $S_1$.
\end{theorem}
\begin{theorem} \label{thm3}
Let $S$ be a closed sector in $\C$ of opening angle less than $\pi/3$ and let 
$H$ be a closed half-plane intersecting $S$ only in~$0$.
Let $f$ be a transcendental entire function of finite order.
Suppose that all but finitely many zeros of $f$ are in $S$ while all but finitely many
$1$-points are in~$H$.
Then $f$ has the form $f(z)=P(z)e^{az}$ where $P$ is a polynomial and $a\in\C$.
\end{theorem}
The following examples show that the 
hypotheses of Theorems~\ref{thm1} and~\ref{thm3} are sharp.
We will verify in Section~\ref{examples} that these examples have the stated properties.

Our first example shows that the condition $\min\{\theta_0,\theta_1\}<\pi/2$ in Theorem~\ref{thm1}
cannot be relaxed to $\min\{\theta_0,\theta_1\}\leq\pi/2$.
\begin{example} \label{exa1}
Let
\begin{equation}\label{1a}
f(z):=\frac{2}{\sqrt{\pi}}\int_0^z t^2 \exp\!\left(-t^2\right) dt +\frac12.
\end{equation}
Then all but finitely many zeros of $f$ are in $\{z\colon|\arg z|\leq \pi/4\}$ while 
all but finitely many $1$-points of $f$ are in $\{z\colon|\arg z-\pi|\leq \pi/4\}$.
\end{example}
The exponential function shows that the condition $\max\{\theta_0,\theta_1\}<\pi$ in Theorem~\ref{thm1}
cannot be relaxed to $\max\{\theta_0,\theta_1\}\leq \pi$.
This is also shown by the following example, which has infinitely many zeros and $1$-points.
This example also shows that the conclusion of  Theorem~\ref{thm3} does not hold if $S$ has
opening angle equal to~$\pi/3$.
\begin{example} \label{exa2}
Let $a$ and $b$ be defined by 
\begin{equation}\label{1b}
a\int_0^\infty t^3 \exp\!\left(-t^3\right) dt =\frac13
\quad\text{and}\quad
b\int_0^\infty t \exp\!\left(-t^3\right) dt =\frac13.
\end{equation}
Let
\begin{equation}\label{1c}
f(z):=\int_0^z (at^3 +bt) \exp\!\left(-t^3\right) dt+\frac{1}{3}.
\end{equation}
Then all but finitely many zeros of $f$ are in $\{z\colon|\arg z|\leq \pi/6\}$ while 
all but finitely many $1$-points of $f$ are in $\{z\colon\re z\leq 0\}$.
\end{example}
Theorems~\ref{thm1} and \ref{thm3} will be derived from the following result.
\begin{theorem} \label{thm2}
Let $S_0$ and $S_1$ be closed sectors in $\C$ of opening angles at most~$\pi$.
Suppose that $S_0\cap S_1=\{0\}$ and let $f$ be a transcendental entire function of
finite order for which all but finitely many zeros are in $S_0$ while all but 
finitely many $1$-points are in $S_1$. Then $f$ has the form
\begin{equation}\label{1d}
f(z)=\int_{0}^zp(\zeta)e^{q(\zeta)}d\zeta +c,
\end{equation}
where $p$ and $q$ are polynomials and $c\in\C$.
\end{theorem}
The conclusion of Theorem~\ref{thm2} does not hold if one of the sectors $S_0$ and $S_1$ has 
opening angle greater than~$\pi$. Similarly, the half-plane in Theorem~\ref{thm3}
cannot be replaced by a sector of opening angle greater than~$\pi$.

For example, if $0<\rho<1$ and $(a_n)$ is a sequence of positive real numbers satisfying
$a_n\sim n^{1/\rho}$ as $n\to\infty$, then
\begin{equation}\label{1e}
f(z):=\prod_{n=1}^\infty \left( 1-\frac{z}{a_n}\right) 
\end{equation}
is an entire function with positive zeros for which the $1$-points accumulate at the 
rays $\arg z=\pm\pi(1-1/2\rho)$; see, e.g., \cite[Section~2.5]{Goldberg2008} for the asymptotics
of canonical products with positive zeros.

The restriction on the order is essential in all stated theorems.
In fact, for any two distinct directions there exists an entire function
$f$ of infinite order such that the zeros
of $f$ accumulate in one of these directions while the $1$-points 
accumulate in the other direction; see, e.g., \cite[Theorem~5]{BEH}.

On the other hand, a classical result of Edrei~\cite{Edrei1955} says that if the zeros and
$1$-points of an entire function $f$ lie on finitely many rays, then $f$ has finite order.

\section{Examples}\label{examples}
We consider functions $f$ of the form~\eqref{1d}.  Let 
$d:=\deg(q)$ and let $A$ be the leading coefficient of $q$ so that
$q(z)\sim Az^{d}$ as $z\to\infty$. For $k\in\{1,2,\dots,d\}$ we put 
\begin{equation}\label{3a}
\phi_k:=\frac{(2k-1)\pi -\arg A}{d} .
\end{equation}
It is easy to see that the limits
\begin{equation}\label{3b}
a_k:=\lim_{r\to\infty} f(re^{i\phi_k})
\end{equation}
exist. 
For $\varepsilon>0$ we then have, as $|z|\to\infty$, 
\begin{equation}\label{3c}
f(z)\to a_k
\quad\text{for}\
\phi_k-\frac{\pi}{2d}+\varepsilon \leq \arg z \leq \phi_k+\frac{\pi}{2d}-\varepsilon ,
\end{equation}
while
\begin{equation}\label{3d}
|f(z)|\to \infty
\quad\text{for}\
\phi_k+\frac{\pi}{2d}+\varepsilon \leq \arg z \leq \phi_{k+1}-\frac{\pi}{2d}-\varepsilon.
\end{equation}
Here we have put $\phi_{d+1}=\phi_1+2\pi$.

The following quantitative form of~\eqref{3c} and~\eqref{3d} can be proved using integration by 
parts; see~\cite[Lemma~4.1]{Hemke2005}.
\begin{lemma} \label{asymp-f}
Let 
\begin{equation}\label{3e}
\phi_k-\frac{\pi}{d} \leq \arg z \leq \phi_k+\frac{\pi}{d}.
\end{equation}
Then
\begin{equation}\label{3f}
f(z)=a_k+\frac{p(z)}{q'(z)}e^{q(z)}\left( 1+\O\!\left(\frac{1}{|z|}\right)\right)
\end{equation}
as $|z|\to\infty$.
\end{lemma}
It is well-known -- and follows easily from Lemma~\ref{asymp-f} -- that 
for any $a\in \C\setminus\{a_k\}$ and $\varepsilon\in (0,\pi/d)$ each of the sectors
$\{z\colon |\arg z-\phi_k\pm \pi/(2d)|<\varepsilon\}$ contains infinitely many $a$-points,
but only finitely many $a_k$-points.

Thus for any $a\in\C$ the $a$-points of $f$
 can accumulate only at the rays given by $\arg z=\phi_k\pm \pi/(2d)$,
and they do accumulate at the rays $\arg z=\phi_k\pm \pi/(2d)$ for $a\in \C\setminus\{a_k\}$, 
but not for $a=a_k$.
\begin{proof}[Verification of Example~\ref{exa1}]
Since
\begin{equation}\label{3g}
\int_0^\infty t^2\exp(-t^2) dt=\frac14\sqrt{\pi}
\end{equation}
we have~\eqref{3b} with $\phi_1=0$, $\phi_2=\pi$, $a_1=1$ and $a_2=0$.
Thus the zeros are asymptotic to the rays $\arg z=\pm \pi/4$ while the $1$-points are 
asymptotic to the rays $\arg z=\pm 3\pi/4$. Lemma~\ref{asymp-f} yields that
$|f(z)|\to\infty$ as $|z|\to\infty$ if $\pi/4\leq \arg z\leq 3\pi/4$ or
$-3\pi/4\leq \arg z\leq -\pi/4$. This implies that all but finitely many zeros 
are contained in $\{z\colon |\arg z|<\pi/4\}$ and all but finitely many $1$-points
are contained in $\{z\colon |\arg z-\pi|<\pi/4\}$.
\end{proof}
\begin{proof}[Verification of Example~\ref{exa2}]
Here we have $d=3$ and~\eqref{3b} holds with $\phi_1=0$, $\phi_2=2\pi/3$, $\phi_3=4\pi/3$,
$a_1=1$ and $a_2=a_3=0$.
Thus the zeros are asymptotic to the rays $\arg z=\pm \pi/6$ while the $1$-points are 
asymptotic to the rays $\arg z=\pm \pi/2$ and $\arg z=\pm 5\pi/6$.
The rest of the argument follows as in Example~\ref{exa1} from Lemma~\ref{asymp-f}.
\end{proof}

\section{Preliminary results}\label{lemmas}
We shall need some results about subharmonic functions; see \cite{Hayman1976},
\cite[Chapter III]{Hoermander1990}, \cite[Chapter III]{Hoermander1994} and~\cite{Ransford1995}
for basic results used in the following.

The following result is implicit in~\cite{BEH}. For completeness, we include a proof.
\begin{lemma} \label{lemma-beh}
Let $u$ be subharmonic function in a neighborhood of $0$ and $u(0)=0$.
Let $\alpha\in (0,\pi]$ and suppose that $u(z)<0$ for $|\arg z|<\alpha$.
Then $\alpha\leq\pi/2$, and there exists $c>0$ and $r_0>0$ such that
\begin{equation}\label{6a0}
\int_{-\alpha}^\alpha u(re^{it}) dt \leq -cr^{\pi/(2\alpha)}
\quad\text{for}\ r\in(0,r_0).
\end{equation}
\end{lemma}
\begin{proof}
Let $r_1>0$ be such that the closed disk of radius $r_1$ around $0$ is in the domain of~$u$.
Then there exists a harmonic majorant  $h$ of $u$ in the (truncated) sector
$S_\alpha:=\{ z\colon |\arg z|<\alpha,|z|<r_1\}$ such that $h(re^{\pm i\alpha})=0$
for $r\in(0,r_1)$.
Let $r_2:=r_1^{\pi/(2\alpha)}$ and $S_{\pi/2}:=\{z\colon \re z> 0, |z|<r_2\}$.
Then the function $v$ defined by $v(z)=h(z^{2\alpha/\pi})$, where the principal branch of the root
is used, is negative and harmonic in $S_{\pi/2}$, and zero on the vertical part of the boundary.
By the reflection principle $v$ extends to
a harmonic function in the disc of radius $r_2$ around $0$, and $\nabla v(0)\neq 0$.
Therefore $v(z)=-c_0\re z+\O(z^2)$ near zero for some $c_0>0$. Thus
\begin{equation}\label{6a0a}
\int_{-\pi/2}^{\pi/2} v(re^{it}) dt \leq -\frac12 c_0 r
\end{equation}
if $r$ is small enough.
This implies that \eqref{6a0} is satisfied with $h$ instead of~$u$.
Since $u<h$ in $S_\alpha$, we deduce that $u$ also satisfies \eqref{6a0}.

To show that $\alpha\leq \pi/2$,
let $S_\alpha'=\{z\colon \alpha<\arg z<2\pi-\alpha, |z|<r_1\}$ be the 
complement of $\overline{S_\alpha}$ in the disk of radius $r_1$ and 
let $h_1$ be a harmonic majorant of $u$ in 
$S_\alpha'$ which satisfies $h_1(re^{\pm i\alpha})=0$ for $r\in (0,r_1)$.
The same argument as before
shows that $h_1(z)=\O(z^{\pi/(2\beta)})$ as $z\to 0$, where $\beta:=\pi-\alpha$.
Thus there exists a constant $c_1>0$ such that
\begin{equation}\label{6a1}
\int_{\alpha}^{2\pi-\alpha} u(re^{it})dt\leq c_1r^{\pi/(2\beta)}
\quad\text{for}\ r\in(0,r_0).
\end{equation}
Together with~\eqref{6a0} this yields that
\begin{equation}\label{6a2}
0=u(0)\leq\int_{\alpha}^{2\pi+\alpha} u(re^{it})dt\leq c_1r^{\pi/(2\beta)}-cr^{\pi/(2\alpha)}
\quad\text{for}\ r\in(0,r_0).
\end{equation}
We conclude that $\beta\geq\alpha$ and thus $\alpha\leq\pi/2.$
\end{proof}
In the case that $\alpha=\pi/2$ we will use the following result.
\begin{lemma} \label{lemma-hp}
Let $u$ be subharmonic in $\C$. Suppose that $u(0)=0$ and that there exist
$\rho>0$ and $K>0$ such that $u(z)\leq K|z|^\rho$ for all $z\in\C$.
If $u$ is negative in some half-plane, then $\rho=1$.
\end{lemma}
In the proof of Lemma~\ref{lemma-hp}
we will use the following version of the Phragm\'en-Lindel\"of theorem~\cite[Corollary~2.3.8]{Ransford1995}.
\begin{lemma} \label{lemma-pl}
Let $u$ be subharmonic in the right half-plane $H:=\{z\colon \re z>0\}$. Suppose that there exist
constants $A,B\in\R$ such that 
\begin{equation}\label{6a}
u(z)\leq A+B|z| \quad\text{for}\ z\in H
\end{equation}
and
\begin{equation}\label{6b}
\limsup_{\zeta\to z}u(\zeta)\leq 0
\quad\text{for}\ z\in \partial H.
\end{equation}
Put
\begin{equation}\label{6c}
L:=\limsup_{x\to \infty}\frac{u(x)}{x}.
\end{equation}
Then
\begin{equation}\label{6d}
u(z)\leq L\re z \quad\text{for}\ z\in H.
\end{equation}
\end{lemma}
\begin{proof}[Proof of Lemma~\ref{lemma-hp}]
Without loss of generality we may assume that $u(z)<0$ for $\re z<0$.
Lemma~\ref{lemma-beh} and the hypotheses that $u(0)=0$ and $u(z)\leq K|z|^\rho$ yield that 
if $r$ is sufficiently small, then
\begin{equation}\label{6d1}
0=2\pi u(0)\leq \int_{-\pi/2}^{3\pi/2} u(re^{it}) dt
=\int_{-\pi/2}^{\pi/2} u(re^{it}) dt +\int_{\pi/2}^{3\pi/2} u(re^{it}) dt
\leq K\pi r^\rho -cr .
\end{equation}
This yields that $\rho\leq 1$. Hence~\eqref{6a} holds with $A=B=K$.

Lemma~\ref{lemma-pl} yields that~\eqref{6d} holds 
with $L$ given by~\eqref{6c}.
Since $u$ is non-constant and hence unbounded we have $L>0$.  Thus $\rho=1$.
\end{proof}

Let  $u$ be subharmonic in $\C$ and let $\mu$ be the Riesz measure of~$u$.
For $r>0$ we put 
\begin{equation}\label{6j}
B(r,u):=\max_{|z|=r} u(z)
\quad\text{and}\quad
n(r):=\mu(\{z\colon |z|\leq r\}.
\end{equation}
Jensen's formula (see~\cite[Section~7.2]{Levin1996} or \cite[Section 3.9]{Hayman1976}) yields that 
\begin{equation}\label{6k}
N(r):=\int_1^r \frac{n(t)}{t} dt\leq B(r,u) +\O(1).
\end{equation}
The order $\rho$ of $u$ is defined by
\begin{equation}\label{6l}
\rho:=\limsup_{r\to\infty} \frac{\log B(r,u)}{\log r} .
\end{equation}
Note that if $u=\log|f|$ for some entire function $f$, 
and if $M(r)$ denotes the maximum modulus of~$f$, then we have $B(r,u)=\log M(r)$.
Thus the order of the subharmonic function $u$ coincides with that of the entire function~$f$.

If $u$ has finite order, then there exists a non-negative integer $q$
satisfying $q\leq \rho\leq q+1$ such that
\begin{equation}\label{6m}
\int_1^\infty \frac{B(r,u)}{r^{q+2}} dr <\infty.
\end{equation}
Using~\eqref{6k} we see that then also
\begin{equation}\label{6n}
\int_1^\infty \frac{N(r)}{r^{q+2}} dr <\infty.
\end{equation}
Moreover, integration by parts shows that the latter condition is equivalent to
\begin{equation}\label{7a}
\int_1^\infty \frac{n(r)}{r^{q+2}} dr <\infty.
\end{equation}
The following result~\cite[Theorem~4.2]{Hayman1976} is the subharmonic version of 
the Hadamard factorization theorem.
\begin{lemma} \label{lemma-hada-subh}
Let $u$ be a subharmonic of finite order $\rho$ with Riesz measure $\mu$.
Let $q$ be the minimal integer such that~\eqref{7a} holds and let $R>0$.
Then $u$ can be written in the form
\begin{equation}\label{7c0}
u(z)=v(z)+w(z)+h(z)
\end{equation}
where
\begin{equation}\label{7c}
v(z) =\int_{|\zeta|<R}\log|z-\zeta|d\mu(\zeta),
\end{equation}
\begin{equation}\label{7d}
w(z) = \int_{|\zeta|\geq  R}\left(\log\left|1-\frac{z}{\zeta}\right|
+\sum_{j=1}^q\frac{1}{j}\re\!\left(\frac{z}{\zeta}\right)\right)d\mu(\zeta)
\end{equation}
and $h$ is a harmonic polynomial of degree at most~$\rho$.
\end{lemma} 
We also note that the function $w$ defined by~\eqref{7d} satisfies
\begin{equation}\label{7d0}
B(r,w)=o(r^{q+1})
\end{equation}
as $r\to\infty$. This follows easily from~\cite[Lemma~4.4]{Hayman1976}; see also~\cite[p.~57, Remark~2]{Goldberg2008}.
\begin{lemma} \label{lemma-unb-subh}
Let $u$ be a subharmonic function of order at most~$1$. 
Suppose that there exists $R>0$ and $\varepsilon>0$ such that $u$ is harmonic
in each of the two domains $T_{\pm}:=\{z\colon |z|>R, |\arg z\pm \pi/2|<\varepsilon\}$.
Suppose also that $u$ is bounded on the imaginary axis.
Then $u$ is harmonic and of the form $u(z)=a\re z+ b$ with $a,b\in\R$.
\end{lemma}
\begin{proof}
We apply Lemma~\ref{lemma-hada-subh}.
Since $u$ has order $1$, we have $q\in\{0,1\}$ in~\eqref{7d} and $h$ has the form 
\begin{equation}\label{7d1}
h(z) = \re(Az+B)
\end{equation}
for some $A,B\in\C$. For $y\in\R$ we then have
\begin{equation}\label{7d2}
h(iy) +h(-iy) = \re(Aiy+B)+\re(-Aiy+B)=2\re B. 
\end{equation}
Assuming that $u$ is not harmonic,
we may choose $R$ such that $\mu(\{z\colon |z|<R\})$ is positive and we find that
$v(z)\to\infty$ as $|z|\to\infty$. In particular, 
\begin{equation}\label{7d3}
v(iy) +v(-iy)\to\infty
\end{equation}
as $y\to\infty$.
Since $u$ is bounded on the imaginary axis, \eqref{7c0}, \eqref{7d2} and~\eqref{7d3} imply that
\begin{equation}\label{7d4}
Q(y):=w(iy) +w(-iy)\to-\infty
\end{equation}

First we consider the case $q=0$.
Define $w^*$ by $w^*(z)=w(z)+w(\overline{z})$.  Then $B(r,w^*)=o(r)$ as $r\to\infty$ by~\eqref{7d0}.
Moreover, $w^*(iy)=Q(y)$ so that $w^*$ is bounded on the imaginary axis by~\eqref{7d4}.
The Phragm\'en-Lindel\"of theorem (Lemma~\ref{lemma-pl}) now implies that $w^*$ is constant,
contradicting~\eqref{7d4}.

Now we consider the case $q=1$. In order to estimate $Q(y)$ in this case we note that
\begin{equation}\label{7e}
\begin{aligned}
&\quad \log\left|1-\frac{iy}{\zeta}\right|+\re\!\left(\frac{iy}{\zeta}\right)
+\log\left|1-\frac{-iy}{\zeta}\right|+\re\!\left(\frac{-iy}{\zeta}\right)
\\ &
=\log\left|1-\frac{iy}{\zeta}\right|
+\log\left|1-\frac{iy}{\overline{\zeta}}\right|
\\ &
=\log\left|1-i\frac{2y\re\zeta }{|\zeta|^2}-\frac{y^2}{|\zeta|^2}\right|
\\ &
=\frac12\log\!\left(\left(1-\frac{y^2}{|\zeta|^2}\right)^2+4\frac{y^2(\re\zeta)^2 }{|\zeta|^4}\right)
\end{aligned}
\end{equation}
so that
\begin{equation}\label{7e1}
Q(y)=\frac12\int_{|\zeta|\geq R}
\log\!\left(\left(1-\frac{y^2}{|\zeta|^2}\right)^2+4\frac{y^2(\re\zeta)^2 }{|\zeta|^4}\right) d\mu(\zeta).
\end{equation}
For $|\arg \zeta\pm\pi/2|\geq\varepsilon$ we have
$|\re \zeta|\geq \alpha|\zeta|$ with $\alpha:=\cos(\pi/2-\varepsilon)>0$
and thus
\begin{equation}\label{7f}
\begin{aligned}
\log\!\left(\left(1-\frac{y^2}{|\zeta|^2}\right)^2+4\frac{y^2(\re\zeta)^2 }{|\zeta|^4}\right)
&\geq\log\!\left(\left(1-\frac{y^2}{|\zeta|^2}\right)^2+4\alpha^2\frac{y^2}{|\zeta|^2}\right)
\\ &
=\log\!\left(1+\left(4\alpha^2-2\right) \frac{y^2}{|\zeta|^2}+\frac{y^4}{|\zeta|^4}\right) .
\end{aligned}
\end{equation}
As $u$ is harmonic in $T_\pm$ we find that~\eqref{7f} holds
for every $\zeta$ in the support of $\mu$ which satisfies $|\zeta|\geq R$.
We put $n_R(r)=\mu\{z\colon R\leq |z|\leq r\}$ and note that $n_R(r)=n(r)-\mu(\{z\colon |z|<R\})$;
that is, $n_R(r)$ and $n(r)$ differ only by a constant.
We deduce from~\eqref{7e1} and~\eqref{7f}, using also integration by parts, that
\begin{equation}\label{7g}
\begin{aligned}
Q(y)
&\geq \frac12
\int_R^\infty \log\!\left(1+\left(4\alpha^2-2\right) \frac{t^2}{|\zeta|^2}+\frac{t^4}{|\zeta|^4}\right)dn(t)
\\ &
= 2\int_0^\infty n_R(t) K\!\left(\frac{y}{t}\right) \frac{dt}{t} ,
\end{aligned}
\end{equation}
where
\begin{equation}\label{5h}
K(x):=\frac{x^2\left(2\alpha^2-1+x^2\right)}{1+\left(4\alpha^2-2\right)x^2+x^4}
=\frac{x^2\left(\beta+x^2\right)}{1+2\beta x^2+x^4}
\end{equation}
with $\beta:=2\alpha^2-1=2\cos^2(\pi/2-\varepsilon)-1=\cos(\pi-2\varepsilon)$.
It is easy to see that there exists a constant $c_1$ depending only on $\alpha$ such that
\begin{equation}\label{5k}
|K(x)|\leq c_1 x^2
\quad\text{for}\ x\geq 0.
\end{equation}
In fact, this holds with $c_1:=1/(4\alpha\sqrt{1-\alpha^2})$. Moreover, we have
\begin{equation}\label{5k1}
|K(x)|\leq \frac{4x^2}{2+x^4}
\quad\text{for}\ x\geq 2.
\end{equation}

By~\eqref{7d4} there exists $y_0>0$ such that $Q(y)\leq 0$ for $y\geq y_0$.
For $0<\delta<1$ we thus have
 \begin{equation}\label{5i}
\begin{aligned}
0
&\geq  \int_{y_0}^\infty \frac{1}{y^{2+\delta}} Q(y) dy 
\\ &
= \int_{y_0}^\infty  \frac{1}{y^{2+\delta}}\int_0^\infty n_R(t) K\!\left(\frac{y}{t}\right)\frac{dt}{t} dy
\\ &
=\int_{0}^\infty n_R(t) \int_{y_0}^\infty  \frac{1}{y^{2+\delta}}K\!\left(\frac{y}{t}\right)\frac{1}{t} dy \;dt
\\ &
=\int_{0}^\infty  \frac{n_R(t)}{t^{2+\delta}} \int_{y_0/t}^\infty \frac{K(s)}{s^{2+\delta}} ds \;dt .
\end{aligned}
\end{equation}
Here we have changed the order of integration. This is justified by the Fubini-Tonelli theorem, since
by~\eqref{5k} and~\eqref{5k1} the above integrals are finite if $K(\cdot)$ is replaced by $|K(\cdot)|$.

It follows from~\eqref{5k} that
\begin{equation}\label{5l}
 \int_0^{y_0/t} \frac{K(s)}{s^{2+\delta}} ds
\leq c_1 \int_0^{y_0/t} \frac{ds}{s^{\delta}} 
=\frac{c_1}{1-\delta}\left(\frac{y_0}{t}\right)^{1-\delta}
\end{equation}
and hence
\begin{equation}\label{5m}
\int_{0}^\infty \frac{n_R(t)}{t^{2+\delta}} \int_0^{y_0/t} \frac{K(s)}{s^{2+\delta}} ds \;dt 
\leq 
c_2:=\frac{c_1 y_0^{1-\delta}}{1-\delta}
\int_{0}^\infty  \frac{n_R(t)}{t^{3}} dt <\infty
\end{equation}
by~\eqref{7a}. Combining this with~\eqref{5i} we deduce that
\begin{equation}\label{5n}
\int_{0}^\infty \frac{n_R(t)}{t^{2+\delta}} dt \cdot \int_0^{\infty} \frac{K(s)}{s^{2+\delta}} ds 
\leq c_2
\end{equation}
for some constant $c_2$. Note that $c_2$ remains bounded as $\delta\to 0$.

We have
\begin{equation}\label{5n1}
\int_{0}^\infty \frac{K(s)}{s^{2+\delta}} ds 
=\frac12  \int_{0}^\infty x^{-\gamma} \frac{x+\beta}{1+2\beta x+x^2} dx
\end{equation}
with $\gamma:=(1+\delta)/2$.
The computation of the integral on the right hand side is a standard 
application of the residue theorem which yields that
\begin{equation}\label{5j}
 \int_{0}^\infty \frac{K(s)}{s^{2+\delta}} ds =
\frac12 \frac{\pi}{\sin(\pi\gamma)}\cos(\gamma(\pi-\arccos(-\beta)))
= \frac12 \frac{\pi}{\sin(\pi\gamma)}\cos(\gamma(\pi-2\varepsilon)).
\end{equation}
Since $\gamma\to 1/2$ as $\delta\to 0$ we thus have
\begin{equation}\label{5j1}
\lim_{\delta\to 0} \int_{0}^\infty \frac{K(s)}{s^{2+\delta}} ds
= \frac12 \frac{\pi}{\sin(\pi/2)}\cos\!\left(\frac12 \pi-\varepsilon\right)  =\frac12\pi\alpha >0.
\end{equation}
Moreover,
\begin{equation}\label{5o}
\int_{0}^\infty \frac{n_R(t)}{t^{2}} dt = \infty,
\end{equation}
since $q=1$, and $q$ is chosen as the smallest integer such that~\eqref{7a} holds. Thus
\begin{equation}\label{5o1}
\lim_{\delta\to 0}  \int_{0}^\infty \frac{n_R(t)}{t^{2+\delta}} dt = \infty.
\end{equation}
Taking the limit as $\delta\to 0$ in~\eqref{5n} yields a contradiction to~\eqref{5j1} and~\eqref{5o1}.
\end{proof}

\section{Proofs of the Theorems}\label{proofs}
\begin{proof}[Proof of Theorem~\ref{thm2}]
Let $\theta_j$ be the opening angle of $S_j$. Since only one of the two
sectors can have opening angle less than~$\pi$, 
we may assume without loss of generality that $\theta_0<\pi$. 
We may also assume that $f(0)\notin\{0,1\}$. 

First we show that the genus of $f$ is at least~$1$.
Suppose that this is not the case so that $f$ is of genus~$0$. Then $f$ has the form
\begin{equation}\label{2a}
f(z)=f(0)\lim_{n\to\infty} \prod_{k=1}^n \left(1-\frac{z}{a_k}\right)
=1+(f(0)-1)\lim_{n\to\infty} \prod_{k=1}^n \left(1-\frac{z}{b_k}\right),
\end{equation}
with all but finitely many $a_k$ contained in $S_0$  and all but finitely many
$b_k$ contained in~$S_1$. The Gauss-Lucas theorem and Hurwitz's theorem imply that
all zeros of $f'$ are contained in the convex hull of the $a_k$ as well as in the convex
hull of the $b_k$. Thus $f'$ has only finitely many zeros.
Since we assumed that $f$ has genus $0$, this contradicts our hypothesis that $f$ 
is transcendental. Hence $f$ has genus at least~$1$.

Let $\rho$ be the order of~$f$.  Since $f$ has genus at least~$1$ we have $\rho\geq 1$.
Let $M(r)$ denote the maximum modulus of~$f$.
We proceed as in~\cite{BEH} and note that since $f$ is of finite order,
there exists a sequence  $(r_k)$ tending to $\infty$ such that 
\begin{equation} \label{A}
\log M(tr_k)=O(\log M(r_k))\quad \text{as}\ k\to\infty,
\end{equation}
for every $t>1$. The existence of such a sequence $(r_k)$ for a function $f$ of finite order is
well-known and easy to prove, but it is also an immediate consequence 
of a result of Drasin and Shea~\cite{Drasin1972} on the existence of P\'olya peaks.
Put
\begin{equation}\label{rho1}
\rho^*:=\sup\left\{ p\in {\mathbb R}  \colon  \limsup_{r,t\to\infty}
\frac{\log M(tr)}{t^p\log M(r)}=\infty\right\}
\end{equation}
and
\begin{equation}\label{rho2}
\rho_*:=\inf\left\{ p \in {\mathbb R}  \colon  \liminf_{r,t\to\infty}
\frac{\log M(tr)}{t^p\log M(r)}=0\right\} .
\end{equation}
The result of Drasin and Shea~\cite{Drasin1972} says that 
\begin{equation}\label{rr}
0\leq\rho_*\leq\rho\leq\rho^*\leq\infty
\end{equation}
and if $\lambda\in\R$ satisfies $\rho_*\leq\lambda\leq\rho^*$,
then $\log M(r)$ has a sequence of P\'olya peaks of order $\lambda$; that is, there exists
a sequence $(r_k)$ tending to $\infty$ such that if $\varepsilon>0$, then
\begin{equation}\label{pp}
\log M(tr_k)\leq (1+\varepsilon)t^{\lambda}\log M(r_k)
\quad\text{for}\  \varepsilon\leq t\leq
\frac{1}{\varepsilon},
\end{equation}
provided $k$ is large enough.  

For a sequence $(r_k)$ satisfying~\eqref{A} we consider, as in~\cite{BEH},
the subharmonic functions
\begin{equation}\label{duv}
u_k(z) := \frac{\log|f(r_kz)|}{\log M(r_k)}
\quad\text{and}\quad
v_k(z) :=  \frac{\log|f(r_kz)-1|}{\log M(r_k)}.
\end{equation}
Arguing as in~\cite[p.~97]{BEH} we can deduce from 
\cite[Theorems 4.1.8 and 4.1.9]{Hoermander1990} or
\cite[Theorems 3.2.12 and 3.2.13]{Hoermander1994} that,
passing to a subsequence if necessary, the limits
\begin{equation}\label{lim}
u(z):=\lim_{k\to\infty}\frac{\log|f(r_kz)|}{\log M(r_k)}
\quad\text{and}\quad
v(z):=\lim_{k\to\infty}\frac{\log|f(r_kz)-1|}{\log M(r_k)}
\end{equation}
exist and are subharmonic in $\C$.
Here the convergence is in the Schwartz space~$\mathscr{D}'$.
This implies that we also have convergence in $L^1_{\mathrm{loc}}$.

Moreover, the limits in~\eqref{lim} have the following properties:
\begin{enumerate}
\item[$(a)$]
$\max\{u(z),0\}=\max\{v(z),0\}$  for all $z\in \C$;
\item[$(b)$]
$\{ z \colon  u(z)<0\}\cap\{ z \colon  v(z)<0\}=\emptyset$;
\item[$(c)$]
$u$ is harmonic in $\C\backslash S_0$ and $v$ is harmonic in $\C\backslash S_1$;
\item[$(d)$]
$\max_{|z|=1}u(z) =\max_{|z|=1}v(z)=1$.
\end{enumerate}
If $(r_k)$ is a sequence of P\'olya peaks of order~$\lambda>0$, then we also have
\begin{enumerate}
\item[$(e)$]
$u(0)=v(0)=0$;
\item[$(f)$]
$\max\{u(z),v(z)\}\leq |z|^\lambda$ for all $z\in \C$.
\end{enumerate}
We refer to~\cite[p.~97]{BEH} for the deduction of these properties.

Suppose first that $(r_k)$ is a sequence of P\'olya peaks of order $\lambda>0$ 
so that $(a)$--$(f)$ hold. Note that such a sequence exists by~\eqref{rr} since $\rho\geq 1$.

Let $P$ be the set where one of the functions $u$ and $v$ is positive.
In view of $(a)$ this set coincides with the set where both functions are positive.
Let $z_0\in P$. Then $z_0\neq 0$ by $(e)$ and since $S_0\cap S_1=\{0\}$ we deduce
from $(c)$ that one of the functions $u$ and $v$ is harmonic in 
some neighborhood of~$z_0$. In particular,
it is continuous in $z_0$ and thus positive in some (possibly smaller) neighborhood of~$z_0$.
Hence $P$ is open.

Let $N$ be the set of points where at least one of the functions $u$ and $v$ is negative.
Since subharmonic functions are upper semicontinuous, $N$ is also open.

Thus the complement $E:=\C\setminus (P\cup N)=\{z\colon u(z)=v(z)=0\}$ is closed. We will show that $E$ has
no interior. Suppose to the contrary that $E$ has an interior point $z_0$.
Without loss of generality we may assume that $z_0\notin S_1$. By $(c)$, 
the function $v$ is harmonic in $\C\setminus S_1$, which contains a
neighborhood of~$z_0$. It follows that $v(z)=0$ for all $z\in \C\setminus S_1$.
By $(a)$ we have $u(z)\leq 0$ for all $z\in \C\setminus S_1$.
If $u(z_1)<0$ for some $z_1\in  \C\setminus S_1$, then 
$u(z)<0$ for all $z\in  \C\setminus S_1$ by the maximum principle.
This is a contradiction since $u(z_0)=0$ and we assumed that $z_0\in\C\setminus S_1$.
Hence $u(z)=0$ for all $z\in \C\setminus S_1$. Since $u$ is harmonic in $\C\setminus S_0$
by $(c)$, we find that $u(z)= 0$ for all $z\in\C$, a contradiction.
Thus $E$ has no interior.

Our next goal is to show that either $N\subset S_0\cup S_1$ or $N\supset \C\setminus S_1$,
where the latter case can occur only if $\theta_1=\pi$. To this end,
let $Q$ be a component of $N$ such that $u(z)<0$ for $z\in Q$.
We will show that $\partial Q\subset S_1$. In order to do so, suppose that $z_0\in \partial Q\backslash S_1$.
Then $v$ is harmonic in some neighborhood $V$ of $z_0$. 
By $(a)$ and $(b)$ we have $v(z)=0$ for $z\in Q$. It follows that $v(z)=0$ for $z\in V$.
On the other hand, $V$ also contains a point $z_1$ such that $u(z_1)>0$.
Thus $v(z_1)=u(z_1)>0$ by~$(a)$. This is a contradiction. Thus $\partial Q\subset S_1$. 
If $\theta_1<\pi$, then $Q$ cannot contain $\C\setminus S_1$ by Lemma~\ref{lemma-beh}.
Hence $Q\subset S_1$ if $\theta_1<\pi$.
But if $\theta_1=\pi$, then it is also possible that $Q\supset \C\setminus S_1$.

Similarly, if $Q$ is a component of $N$ such that $v(z)<0$ for $z\in Q$,
then $\partial Q\subset S_0$. Since we assumed that $\theta_0<\pi$, this implies 
together with Lemma~\ref{lemma-beh} that $Q\subset S_0$. 
Overall we see that $N\subset S_0\cup S_1$ or $N\supset \C\setminus S_1$, as claimed above.
In the latter case, $u(z)<0$ for $z\in \C\setminus S_1$.

We first consider the case that $N\subset S_0\cup S_1$.
It follows that $u(z)\geq 0$ and $v(z)\geq 0$ for $z\in \C\setminus (S_0\cup S_1)$.
Since both $u$  and $v$ are harmonic in $\C\setminus (S_0\cup S_1)$ 
we actually have $u(z)>0$ and $v(z)>0$ for $z\in\C\setminus (S_0\cup S_1)$
by the minimum principle. Hence $u(z)=v(z)$ for $z\in \C\setminus (S_0\cup S_1)$ by $(a)$. 
Thus the function
\begin{equation}\label{defw}
w(z):=
\begin{cases}
u(z)& \text{if} \ z\in\C\backslash S_0,\\
v(z)& \text{if} \ z\in\C\backslash S_1,
\end{cases}
\end{equation}
is well-defined and harmonic in $\C\backslash\{0\}$. By the removable
singularity theorem~\cite[Theorem~2.3]{Axler1992} it is harmonic in $\C$. 
Hence $w$ has the form $w=\re g$ for some entire function~$g$.
Since $w(0)=0$ we may choose $g$ such that $g(0)=0$.

Since $(r_k)$ is a sequence of P\'olya peaks of order~$\lambda>0$, 
we can deduce from $(f)$ that 
\begin{equation}\label{estw1}
w(z)=\re g(z)\leq |z|^\lambda.
\end{equation}
This implies that $g$ is a polynomial. In fact, $\lambda$ is a positive integer and we have
\begin{equation}\label{formg}
g(z)=cz^\lambda
\end{equation}
for some $c\in\C$. Moreover, $|c|=1$ by~$(d)$.

In the above reasoning we can take any $\lambda\in[\rho_*,\rho^*]\cap (0,\infty)$,
but the conclusion gives that $\lambda$ is a positive integer. Together with~\eqref{rr}
we conclude that 
\begin{equation}\label{rhos}
\rho_*=\rho^*=\rho\in\N.
\end{equation}
Thus the only possible value for $\lambda$ is $\lambda=\rho$. It follows from~\eqref{rhos} that~\eqref{A} 
is satisfied for any sequence $(r_k)$ tending to~$\infty$. We mention that this kind of 
argument appears first in~\cite[Section~7]{Eremenko1989} and~\cite[p.~1209]{Eremenko1993},
and it was also used in~\cite[p.~100]{BEH}.

More precisely, \eqref{rhos} implies that for any $\delta>0$ there exist $r_0,t_0>0$ such that
\begin{equation}\label{2c}
t^{\rho-\delta}\log M(r) \leq \log M(tr)\leq t^{\rho+\delta}\log M(r) 
\quad\text{for}\  r\geq r_0 \ \text{and}\  t\geq t_0.
\end{equation}
Dropping the assumption that $(r_k)$ is a sequence of P\'olya peaks we 
still have properties $(a)$--$(d)$, but instead of $(f)$ we can deduce from~\eqref{2c} only that
\begin{enumerate}
\item[$(f')$]
$\max\{u(z),v(z)\}\leq
\begin{cases}
|z|^{\rho+\delta} \quad\text{for}\ |z|\geq t_0,\\
|z|^{\rho-\delta} \quad\text{for}\ |z|\leq 1/t_0.
\end{cases}
$
\end{enumerate}
This still yields~$(e)$.

Still assuming that the set $N$ where one of the functions $u$ and $v$ is negative
is contained in $S_0\cup S_1$,
we again find that the function $w$ defined by~\eqref{defw} is harmonic and of 
the form $w=\re g$ for some entire function~$g$.
Instead of~\eqref{estw1}, which was obtained from~$(f)$, we now deduce from~$(f')$ that
\begin{equation}\label{estw2}
w(z)=\re g(z)\leq 
\begin{cases}
|z|^{\rho+\delta} \quad\text{for}\ |z|\geq t_0,\\
|z|^{\rho-\delta} \quad\text{for}\ |z|\leq 1/t_0.
\end{cases}
\end{equation}
This implies that $g$ is a polynomial of degree at most $\rho+\delta$ which has a zero 
of multiplicity at least $\rho-\delta$ at the origin.
Choosing $\delta<1$ we again find that $g$ has the form~\eqref{formg} with $\lambda=\rho$; that is, 
\begin{equation}\label{formg1}
g(z)=cz^\rho.
\end{equation}

To summarize, every sequence tending to $\infty$ has a subsequence $(r_k)$ such that 
\begin{equation}\label{lim1}
\lim_{k\to\infty}\frac{\log|f(r_kz)|}{\log M(r_k)}
=\re (cz^\rho)
\quad\text{for}\ z\in\C\setminus S_0
\end{equation}
while
\begin{equation}\label{lim2}
\lim_{k\to\infty}\frac{\log|f(r_kz)-1|}{\log M(r_k)}
=\re (cz^\rho)
\quad\text{for}\ z\in\C\setminus S_1.
\end{equation}
Next we note that $\log|f|=\re(\log f)$ for a branch  $\log f$ of the logarithm.
Also, the derivative $h'$ of a holomorphic function $h$ can be computed from its real part,
via $h'=\partial \re h/\partial x-i\partial\re h/\partial y$, or via
\begin{equation}\label{schwarz1}
h'(z)=\frac{1}{\pi i} \int_{|\zeta-a|=t}\frac{\re h(\zeta)}{(\zeta-z)^2}d\zeta
\quad\text{for}\ |z-a|<r 
\end{equation}
if $h$ is holomorphic in $\{z\colon |z-a|\leq r\}$. We can thus deduce from~\eqref{lim1} that
\begin{equation}\label{lim1a}
\lim_{k\to\infty}\frac{r_kf'(r_kz)}{f(r_kz) \log M(r_k)} =c\rho z^{\rho-1}
\quad\text{for}\ z\in\C\setminus S_0.
\end{equation}
In particular, if $T_1$ is a closed subsector of $\C\setminus S_1$, then $f'$ has 
only finitely many zeros in~$T_1$. Applying the same argument to~\eqref{lim2} yields
that $f'$ has  only finitely many zeros in any closed subsector $T_0$ of $\C\setminus S_0$.
As we may choose $T_0$ and $T_1$ such that $T_0\cup T_1=\C$ we conclude that $f'$ has
only finitely many zeros in~$\C$. Since $f$ and hence $f'$ have finite order this implies
that $f'$ has the form $f'=pe^q$ with polynomials $p$ and~$q$.
Thus $f$ has the form~\eqref{1d}.

It remains to consider the case that  $N\supset \C\setminus S_1$,
with $\theta_1=\pi$ and $u(z)<0$ for $z\in\C\setminus S_1$. We may assume without 
loss of generality that $S_1$ is the left half-plane and $S_0=\{z\colon |\arg z|\leq \pi-\varepsilon\}$
for some $\varepsilon>0$. It follows from $(e)$ and $(f)$ that $u$ satisfies 
the hypotheses of Lemma~\ref{lemma-hp}. This lemma then  yields that $\rho=1$.
Since $u$ is harmonic in $C\setminus S_0$ by~$(c)$, Lemma~\ref{lemma-unb-subh} yields
that $u$ has the form $u(z)=az+b$.
Since $u(0)=0$ we have $b=0$ and using $(d)$ we see that $|a|=1$ and in fact $a=-1$.
Hence $u(z)=-\re(z)$; that is,
\begin{equation}\label{5a0}
u_k(z) = \frac{\log|f(r_kz)|}{\log M(r_k)} \to -\re z .
\end{equation}
As before we can now deduce that the limits~\eqref{lim} exist not only 
if $(r_k)$ is chosen as a sequence of P\'olya peaks, but 
in fact for every sequence $(r_k)$ tending to $\infty$.
Once this is known, it is not difficult to see that the question whether $N\subset S_0\cup S_1$
or $N\supset \C\setminus S_1$ does not depend on the choice of the sequence~$(r_k)$.

Thus we only have to deal with the case that~\eqref{5a0} holds
for every sequence $(r_k)$ tending to $\infty$.
We will show that this implies that $f$ has the form $f(z)=e^{az+b}$ with constants $a$ and~$b$.
In particular, $f$ has the form~\eqref{1d}.

It follows from~\eqref{5a0} that
there is a curve $\gamma$ tending to $\infty$ near the imaginary axis in
both directions such that $|f(z)|=1$ for $z\in\gamma$. Suppose that $f$ has at least one zero.
Then $f$ is unbounded on the imaginary axis by Lemma~\ref{lemma-unb-subh}, applied
to the subharmonic function $\log|f|$. Thus there exists a real sequence $(y_k)$ such that
\begin{equation}\label{5p}
T_k:=|f(iy_k)|\to\infty
\end{equation}
as $k\to\infty$. Without loss of generality we may assume that $y_k\to+\infty$.
Assuming that $T_k>1$ there exists $x_k>0$ such that $z_k:=x_k+iy_k$ lies on the curve~$\gamma$.
We have $x_k=o(y_k)$ as $k\to\infty$ by~\eqref{5a0}.

As before it follows from~\eqref{5a0} with $r_k=y_k$ by differentiation that
\begin{equation}\label{5r}
\lim_{k\to\infty}\frac{y_kf'(y_kz)}{f(y_kz) \log M(y_k)} = -1
\quad\text{for}\ z\in\C\setminus S_0.
\end{equation}
Put $L_k:=(\log M(y_k))/y_k$. It follows from~\eqref{5r} that 
\begin{equation}\label{5s}
\frac12 L_k \leq \left|\frac{f'(z)}{f(z)}\right|\leq 2 L_k
\quad\text{for}\ |z-iy_k|\leq 2x_k,
\end{equation}
provided $k$ is large enough. Hence 
\begin{equation}\label{5t}
\begin{aligned}
\log T_k 
&=\log|f(iy_k)|-\log|f(x_k+iy_k)|
\\ &
= \re\!\left(-\int_0^{x_k}\frac{f'(x+iy_k)}{f(x+iy_k)}dx\right)
\\ &
\leq \int_0^{x_k}\left|\frac{f'(x+iy_k)}{f(x+iy_k)}\right|dx
\\ &
\leq 2 x_k L_k .
\end{aligned}
\end{equation}
Let now $\gamma_k$ be the component of the intersection of $\gamma$ with the disk 
$\{z\colon |z-z_k|\leq x_k\}$ that contains $z_k$.
Then $f\circ \gamma_k$ is a curve contained in the unit circle.
We have $f'(z)\neq 0$ for $z\in\gamma_k$ by~\eqref{5s}, provided $k$ is large enough.
This implies that $\arg f(z)$ is monotone as $z$ runs through~$\gamma_k$. Moreover, we have
\begin{equation}\label{5u}
\length(\gamma_k)\geq 2x_k
\end{equation}
and thus
\begin{equation}\label{5v}
\length(f\circ\gamma_k) \geq 2x_k \inf_{z\in\gamma_k}|f'(z)|
= 2x_k \inf_{z\in\gamma_k}\left|\frac{f'(z)}{f(z)}\right| 
\end{equation}
for large~$k$.
Combining this with~\eqref{5s}, \eqref{5t} and~\eqref{5p} we find that
\begin{equation}\label{5w}
\length(f\circ\gamma_k) \geq x_k L_k\geq \frac12 \log T_k > 2\pi
\end{equation}
for large~$k$.
We conclude that $f\circ\gamma_k$ wraps around the unit circle at least once.
Hence $\gamma_k$ contains a $1$-point of $f$, contradicting the hypothesis that
all $1$-points are in the left half-plane.

It follows that $f$ has no zeros. Hence $f$ is of the form $f(z)=e^{az+b}$.
\end{proof}

\begin{proof}[Proof of Theorem~\ref{thm1}]
Let $f$ be a transcendental entire function for which all but finitely many zeros are
in $S_0$ while all but finitely many $1$-points are in~$S_1$.
Theorem~\ref{thm2} yields that $f$ has the form~\eqref{1d} with polynomials $p$ and~$q$.
We show first that the degree of $q$ is even.
To see this, let $\phi_k$ and $a_k$ be given by~\eqref{3a} and~\eqref{3b}
and suppose that $d:=\deg(q)$ is odd, say $d=2m-1$ with $m\in\N$.
First we note that $a_k\in\{0,1\}$ for all $k\in\{1,\dots,d\}$ since otherwise both 
zeros and $1$-points would accumulate at the rays $\arg z=\phi_k\pm\pi/(2d)$.

Now fix $k\in\{1,\dots,d\}$ and suppose that $a_k=0$. Then the $1$-points of $f$ accumulate at 
the ray $\arg z=\phi_k+\pi/(2d)$.
Since all but finitely many $1$-points are contained in a sector of opening 
less than~$\pi$, they cannot accumulate at the ray $\arg z=\phi_k+\pi/(2d)+\pi=\phi_{k+m}-2\pi/(2d)$.
Here the index in $\phi_k$ is taken modulo~$d$; that is, $\phi_j=\phi_k$ if $j\equiv k\pmod d$.
Hence the $1$-points also do not accumulate at the ray $\arg z=\phi_{k+m}+2\pi/(2d)$.
It follows that the zeros accumulate at the ray $\arg z=\phi_{k+m}+2\pi/(2d)$.
Repeating this argument we deduce that the $1$-points accumulate at the rays
$\arg z=\phi_{k+2m}\pm2\pi/(2d)=\phi_{k+1}\pm2\pi/(2d)$.
Induction now shows that the $1$-points accumulate at the rays 
$\arg z=\phi_{j}\pm2\pi/(2d)$ for all $j\in\{1,\dots,d\}$, contradicting 
our hypothesis that they are contained in a sector of opening angle less than~$\pi$.
Thus $d$ is even.

Suppose first that $d=2$. Without loss of generality we may assume that $\theta_0<\pi/2$.
Then, for $k\in\{1,2\}$, the zeros of $f$ cannot  accumulate at both rays $\arg z=\phi_k\pm\pi/4$.
As explained in Section~\ref{examples}, this implies that the $1$-points of $f$ accumulate at the
rays $\arg z=\phi_1\pm\pi/4$ and $\arg z=\phi_2\pm\pi/4$.
This contradicts the assumption that all but finitely many $1$-points are in~$S_1$.

Suppose now that $d\geq 4$.
It follows from the hypothesis that there exists an open sector $T$ of opening angle greater
than $\pi/4$ such that $f$ has only finitely many zeros and $1$-points in~$T$.
However, since $d\geq 4$, there exists $k\in\{1,\dots,d\}$ such that one of the rays 
$\arg z=\phi_k+\pi/(2d)$ and $\arg z=\phi_k-\pi/(2d)$ is contained in~$T$. Since the
zeros or $1$-points accumulate at these rays, this is a contradiction.
\end{proof}
\begin{proof}[Proof of Theorem~\ref{thm3}]
Let again $\phi_k$ and $a_k$ be given by~\eqref{3a} and~\eqref{3b}, with
$d:=\deg(q)$.
As in the proof of Theorem~\ref{thm1} we have $a_k\in\{0,1\}$ for all $k\in\{1,\dots,d\}$.
By hypothesis there exists a closed sector $T$ of opening angle greater than $\pi/3$ which
intersects $H$ and $S$ only in~$0$.
This implies that $T$ does not intersect any of the rays $\arg z=\phi_{k}\pm \pi/(2d)$.
It follows that $\pi/d>\pi/3$ and thus $d<3$.

Suppose that $d=2$. Since the $1$-points are contained in a half-plane we have $a_k=0$ for
some $k\in\{1,2\}$. This implies that the zeros accumulate at both rays $\arg z=\phi_{k}\pm\pi/4$.
Hence there are infinitely many zeros not contained in~$S$, a contradiction.

It follows that $d=1$. This implies that $f$ has the form given.
\end{proof}

\noindent
Mathematisches Seminar\\
Christian-Albrechts-Universit\"at zu Kiel\\
Ludewig-Meyn-Str.\ 4\\
24098 Kiel\\
Germany\\
{\tt Email: bergweiler@math.uni-kiel.de}

\medskip

\noindent
Department of Mathematics\\
Purdue University\\
West Lafayette, IN 47907\\
USA\\
{\tt Email: eremenko@math.purdue.edu}
\end{document}